\documentclass{article}
\usepackage{amsfonts}
\usepackage{amsmath}

\newtheorem{theorem}{Theorem}[section]

\newtheorem{conjecture}[theorem]{Conjecture}
\newtheorem{corollary}[theorem]{Corollary}

\newtheorem{definition}[theorem]{Definition}

\newtheorem{lemma}[theorem]{Lemma}

\newtheorem{proposition}[theorem]{Proposition}
\newtheorem{remark}[theorem]{Remark}

\newenvironment{proof}[1][Proof]{\noindent\textbf{#1.} }{\ \rule{0.5em}{0.5em}}
\input{tcilatex}
\begin{document}

\title{Partial Euler Characteristic, Normal Generations and the stable $D(2)$
problem}
\author{Feng Ji and Shengkui Ye}
\maketitle

\begin{abstract}
We study the interplay among Wall's $D(2)$ problem, normal generation
conjecture (the Wiegold Conjecture) of perfect groups and Swan's problem on
partial Euler characteristic and deficiency of groups. In particular, for a
3-dimensional complex $X$ of cohomological dimension 2 with finite
fundamental group, assuming the Wiegold conjecture holds, we prove that $X$
is homotopy equivalent to a finite 2-complex after wedging a copy of sphere $%
S^{2}.$
\end{abstract}

\section{Introduction}

In this article, we study several classical problems in low-dimensional
homotopy theory and group theory, focusing on the interplay among these
problems.

Let us first recall the following Swan's problem. Let $G$ be a group and $%
\mathbb{Z}G$ the group ring. Swan \cite{sw} defines the partial Euler
characteristic $\mu _{n}(G)$ as follows. Let $F$ be a resolution%
\begin{equation*}
\cdots \rightarrow F_{2}\rightarrow F_{1}\rightarrow F_{0}\rightarrow 
\mathbb{Z}\rightarrow 0
\end{equation*}%
of the trivial $\mathbb{Z}G$-module $\mathbb{Z}$, in which each $F_{i}$ is $%
\mathbb{Z}G$-free on $f_{i}$ generators. For an integer $n\geq 0,$ if 
\begin{equation*}
f_{0},f_{1},f_{2},\cdots ,f_{n}
\end{equation*}%
are finite, define%
\begin{equation*}
\mu _{n}(F)=f_{n}-f_{n-1}+f_{n-2}-\cdots +(-1)^{n}f_{0}.
\end{equation*}%
If there exists a resolution $F$ such that $\mu _{n}(F)$ is defined, we let $%
\mu _{n}(G)$ be the infimum of $\mu _{n}(F)$ over all such resolutions $F.$
We call the truncated free resolution 
\begin{equation*}
F_{n}\rightarrow \cdots \rightarrow F_{1}\rightarrow F_{0}\rightarrow 
\mathbb{Z}\rightarrow 0
\end{equation*}%
an algebraic $n$-complex if each $F_{i}$ is finitely generated as a $\mathbb{%
Z}G$-module (following the terminology of Johnson \cite{Jo}).

For a finitely presentable group $G,$ the deficiency $\mathrm{def}(G)$ is
the maximum of $d-k$ over all presentations $\langle g_{1},g_{2},\cdots
,g_{d}\mid r_{1},r_{2},\cdots ,r_{k}\rangle $ of $G.$ It is not hard to see
that 
\begin{equation}
\mathrm{def(}G)\leq 1-\mu _{2}(G)  \label{eq1}
\end{equation}%
(\cite{sw}, Proposition 1). However, Swan mentions in \cite{sw} that
\textquotedblleft the problem of determining when $\mathrm{def}(G)=1-\mu
_{2}(G)$ seems very difficult even if $G$ is a finite $p$-group%
\textquotedblright .

Next, we consider Wall's $D(2)$ problem (\textit{cf.} \cite{wa}). The
cohomological dimension $\mathrm{cd}(X)$ of a CW complex $X$ is defined as
the largest integer $n$ such that $H^{n}(X,M)\neq 0$ for some $\mathbb{Z}%
[\pi _{1}(X)]$-module $M$. For a $3$-dimensional CW complex $X$ of
cohomological dimension $\mathrm{cd}(X)=2,$ Wall's $D(2)$ problem asks
whether $X$ is homotopy equivalent to a $2$-dimensional CW complex. A
positive answer to this problem will imply the Eilenberg-Ganea conjecture,
which says that a group of cohomological dimension two has a $2$-dimensional
classifying space. A finitely presentable group $G$ is said to have $D(2)$
property if any finite $3$-dimensional CW complex $X,$ of cohomological
dimension $2$ with fundamental group $G,$ is homotopy equivalent to a $2$%
-dimensional CW complex. For the status of $D(2)$ problem, see Johnson \cite%
{Jo,jo2} (see also \cite{ham, yu} for some recent work).

It is well-known that a finite perfect group $G$ is normally generated by
one element (\cite{lw},4.2). The Wiegold conjecture (\textit{cf.} \cite{bm},
FP14 and \cite{mk}, 5.52) asserts that the same holds for any finitely
generated perfect group:

\begin{conjecture}[Wiegold Conjecture]
\label{ng}Let $G$ be any finitely generated perfect group, \textsl{i.e.} $%
G=[G,G]$, the commutator subgroup of $G$. Then $G$ can be normally generated
by a single element.
\end{conjecture}

Our main result is the following, which gives a relaxed lower bound of $%
\mathrm{def}(G)$ assuming the Wiegold conjecture.

\begin{theorem}
\label{th1}Assume that Conjecture \ref{ng} is true. Let $X$ be a finite
3-dimensional CW complex of cohomological dimension 2 with finite
fundamental group. We have the following:

\begin{enumerate}
\item[(i)] the complex $X$ is homotopy equivalent to a finite 3-dimensional
complex with just one 3-cell;

\item[(ii)] the wedge $X\vee S^{2}$ is homotopy equivalent to a finite
2-dimensional complex;

\item[(iii)] 
\begin{equation*}
1-\mu _{2}(G)\geq \mathrm{def(}G)\geq -\mu _{2}(G)
\end{equation*}%
for any finite group $G.$
\end{enumerate}
\end{theorem}

Our discussions are based on the study of a stable version of the $D(2)$
problem (for details, see Section \ref{sec}). For a group $G$ having a
finite classifying space $\mathrm{B}G$ of dimension at most $2$, we have $%
\mathrm{def}(G)=1-\mu _{2}(G)$, which confirms the equality of partial Euler
characteristic and deficiency (cf. Theorem \ref{lgba} (i)). A famous
conjecture of Whitehead says that any subcomplex of an aspherical
2-dimensional CW complex is aspherical (\textit{cf.} \cite{Bo}). As an
application of the results proved, we reprove the following (\textit{cf.}
Bogley \cite{Bo}).

\begin{corollary}
\label{ncor}A subcomplex $X$ of a finite aspherical 2-dimensional CW complex
is aspherical if and only if the fundamental group $\pi _{1}(X)$ has a
finite classifying space $\mathrm{B}\pi _{1}(X)$ of dimension at most 2.
\end{corollary}

The article is organized as follows. In Section 2, we discuss the Quillen
plus construction of 2-dimensional CW complexes. This motivates the stable
Wall's $D(2)$ property being discussed in Section 3. In the last section,
the Euler characteristics are studied for groups of low geometric dimensions.

\section{Quillen's plus construction of 2-dimensional CW complexes}

Let $X$ be a CW complex with fundamental group $G$ and $P$ a perfect normal
subgroup of $G$, \textsl{i.e.} $P=[P,P].$ Quillen shows that there exists a
CW complex $X_{P}^{+}$, whose fundamental group is $G/P;$ and an inclusion $%
f:X\rightarrow X_{P}^{+}$ such that 
\begin{equation*}
H_{n}(X;f^{\ast }M)\cong H_{n}(X_{P}^{+};M)
\end{equation*}%
for any integer $n$ and local coefficient system $M$ over $X_{P}^{+}.$ Here $%
X_{P}^{+}$ is called the plus-construction of $X$ with respect to $P$. It is
unique up to homotopy equivalence. One of the main applications of the plus
construction is to define higher algebraic $K$-theory. In general, the space 
$X_{P}^{+}$ is obtained from $X$ by attaching 2-cells and 3-cells. We need
the following definition.

\begin{definition}
The cohomological dimension $\mathrm{cd}(X)$ of a CW complex $X$ is defined
as the smallest integer $n$ such that $H^{m}(X,M)=0$ for any integer $m>n$
and any $\mathbb{Z}[\pi _{1}(X)]$-module $M$. If no such $n$ exists, the
cohomological dimension $\mathrm{cd}(X)$ is defined to be $\infty .$
\end{definition}

It is obvious that an $n$-dimensional CW complex is of cohomological
dimension at most $n$. The following well-known lemma gives a property
enjoyed by any $3$-dimensional CW complex with cohomological dimension $2$.

\begin{lemma}
\label{2.2}Suppose that $X$ is a $3$-dimensional CW complex and $\tilde{X}$
is the universal cover of $X$. Let $C_{\ast }(\tilde{X})$ be the cellular
chain complex of $\tilde{X}$. Then $X$ is of cohomological dimension $2$ if
and only if the image of $C_{3}(\tilde{X})$ is a direct summand of $C_{2}(%
\tilde{X})$ as $\mathbb{Z[}\pi _{1}(X)]$-modules.
\end{lemma}

The following result shows that for certain 2-dimensional CW complexes$,$
the Quillen plus construction is homotopy equivalent to a 2-dimensional CW
complex. Let $X$ be a finite 2-dimensional CW complex. Suppose that a
perfect normal subgroup $P$ in $\pi _{1}(X)$ is normally generated by $n$
elements. With respect these normal generators, there is a canonical
construction $Y$ for $X^{+}$ that attaches a $2$-cell bounded by each
generator and a $3$-cell to kill the resulting homology. Moreover, the
number of attached $3$-cells and the number of attached $2$-cells are both $%
n $ (\textit{cf.} the proof of Theorem 1 in \cite{ye}). The cellular chain
complex $C_{\ast }(\tilde{Y})$ of the universal cover $\tilde{Y}$ is 
\begin{eqnarray*}
0 &\rightarrow &\mathbb{Z}[\pi _{1}(Y)]^{n}\rightarrow \mathbb{Z}[\pi
_{1}(Y)]^{n}\tbigoplus C_{2}(\tilde{X})\tbigotimes\nolimits_{\mathbb{Z}[\pi
_{1}(X)]}\mathbb{Z}[\pi _{1}(Y)] \\
&\rightarrow &C_{1}(\tilde{X})\tbigotimes\nolimits_{\mathbb{Z}[\pi _{1}(X)]}%
\mathbb{Z}[\pi _{1}(Y)]\rightarrow C_{0}(\tilde{X})\tbigotimes\nolimits_{%
\mathbb{Z}[\pi _{1}(X)]}\mathbb{Z}[\pi _{1}(Y)]\rightarrow 0,
\end{eqnarray*}%
with the first map the inclusion of a direct summand. This is homotopy
equivalent to%
\begin{eqnarray*}
0 &\rightarrow &C_{2}(\tilde{X})\tbigotimes\nolimits_{\mathbb{Z}[\pi
_{1}(X)]}\mathbb{Z}[\pi _{1}(Y)]\rightarrow C_{1}(\tilde{X}%
)\tbigotimes\nolimits_{\mathbb{Z}[\pi _{1}(X)]}\mathbb{Z}[\pi _{1}(Y)] \\
&\rightarrow &C_{0}(\tilde{X})\tbigotimes\nolimits_{\mathbb{Z}[\pi _{1}(X)]}%
\mathbb{Z}[\pi _{1}(Y)]\rightarrow 0.
\end{eqnarray*}%
It follows that

\begin{lemma}
\label{lxba} The plus construction $(X\vee (S^{2})^{\vee n})^{+}$ of the
wedge of $X$ and $n$ copies of $S^{2},$ taken with respect to $P$, is
homotopy equivalent to the $2$-skeleton of $X^{+}.$
\end{lemma}

\bigskip

The following lemma is from Johnson \cite{Jo} (59.4, p.228). Although the
original version is stated for complexes with finite groups, it does hold
for complexes with finitely presentable groups (\textit{cf.} \cite{Jo},
appendix B and Mannan \cite{ma1}).

\begin{lemma}
\label{ge}Let $Y$ be a finite 3-dimensional CW complex of cohomological
dimension 2. If the reduced chain complex of the universal cover 
\begin{equation*}
0\rightarrow C_{2}(\tilde{Y})/C_{3}(\tilde{Y})\rightarrow C_{1}(\tilde{Y}%
)\rightarrow \mathbb{Z}\pi _{1}(Y)\rightarrow \mathbb{Z}\rightarrow 0
\end{equation*}%
is homotopy equivalent to the chain complex of the universal cover of a $2$%
-dimensional CW complex $X$, then $Y$ is homotopy equivalent to $X$.
\end{lemma}

\section{Wall's D(2) problem and its stable version\label{sec}}

In this section, we apply the results obtained in the previous section to
study the $D(2)$ problem. Let us recall the $D(2)$ problem raised in \cite%
{wa}.

\begin{conjecture}
(The $D(2)$ problem) If $X$ is a finite 3-dimensional CW complex of
cohomological dimension at most $2$, then $X$ is homotopy equivalent to a
2-dimensional CW complex.
\end{conjecture}

In \cite{Jo}, Johnson proposes to systematically study the problem by
parameterizing 3-dimensional CW complexes by their fundamental groups. For a
finitely presentable group $G$, we say the $D(2)$ problem is true for $G$,
if any finite 3-dimensional CW complex $X,$ of cohomological dimension at
most $2$ with fundamental group $\pi _{1}(X)=G,$ is homotopy equivalent to a
2-dimensional CW complex.

The $D(2)$ problem is very difficult in general. It is known to be true for
a limited amount of groups (for an updated state, see \cite{Ed} \cite{Ma}
and \cite{jo2}, p. 261). We propose the following stable version by allowing
taking wedge with copies of $S^{2}$.

\begin{conjecture}
\label{stable}(The $D(2,n)$ problem) Let $n\geq 0$ be an integer. If $X$ is
a finite 3-dimensional CW complex of cohomological dimension at most$\ 2$,
then $X\vee (S^2)^{\vee n}$ is homotopy equivalent to a 2-dimensional CW
complex.
\end{conjecture}

For a finitely presentable group $G$ and an integer $n\geq 0$, we say that $%
G $ has the $D(2,n)$ property (or the $D(2,n)$ problem holds for $G$) if
Conjecture \ref{stable} is true for all those $X$ with fundamental group $G$%
. The $D(2,0)$ problem is the original $D(2)$ problem. It is immediate that
property $D(2)$ implies $D(2,n);$ and $D(2,n)$ implies $D(2,n+1)$ for any
group $G$ and any integer $n\geq 0$.

We now study the relation between the stabilization by wedging\ copies of $%
S^{2}$ with that by attaching\ 3-cells.

\begin{proposition}
\label{prop}Suppose that $X$ is a finite 3-dimensional CW complex of
cohomological dimension at most $2$. Then $X\vee (S^2)^{\vee n}$ is homotopy
equivalent to a finite 2-dimensional CW complex if and only if $X$ is
homotopy equivalent to a 3-dimensional CW complex with $n$ 3-cells.
\end{proposition}

\begin{proof}
Assume that $X$ is homotopy equivalent to a 3-dimensional CW complex $%
X^{\prime }$ with $n$ 3-cells. Denote by $X^{\prime (2)}$ the $2$-skeleton
of $X^{\prime }$ and let $Z=X^{\prime }\vee (S^{2})^{n}.$ It is not hard to
see that the reduced chain complex 
\begin{equation*}
0\rightarrow C_{2}(\tilde{Z})/C_{3}(\tilde{Z})\rightarrow C_{1}(\tilde{Z}%
)\rightarrow \mathbb{Z}\pi _{1}(Z)\rightarrow \mathbb{Z}\rightarrow 0
\end{equation*}%
is homotopy equivalent to the chain complex of the universal cover of $%
X^{\prime (2)}$. By Lemma \ref{ge}, $X\vee (S^{2})^{\vee n}$ is homotopy
equivalent to a 2-dimensional CW complex.

Conversely, suppose that $X\vee (S^2)^{\vee n}$ is homotopy equivalent to a
finite 2-complex $Y$ via a map $f:X\vee (S^2)^{\vee n}\rightarrow Y$. It is
clear that 
\begin{equation*}
\pi _{1}(X)=\pi _{1}(X\vee (S^2)^{\vee n})\cong \pi _{1}(Y).
\end{equation*}%
Let $G=\pi _{1}(X)$ and $\tilde{X},\tilde{Y}$ be the universal covering
spaces of $X,Y$ respectively. By the Hurewicz theorem, we have isomorphisms 
\begin{equation*}
\pi _{2}(Y)\cong \pi _{2}(\tilde{Y})\cong H_{2}(\tilde{Y})\cong \pi _{2}(%
\tilde{X})\oplus \mathbb{Z}G^{n}.
\end{equation*}%
Therefore, there are $n$ maps $f_{i}:S^{2}\rightarrow Y,1\leq i\leq n,$
corresponding to the inclusion onto the second factor (for a fixed basis of $%
\mathbb{Z}G^{n}$) 
\begin{equation*}
\mathbb{Z}G^{n}\rightarrow H_{2}(\tilde{Y})\cong \pi _{2}(\tilde{X})\oplus 
\mathbb{Z}G^{n}.
\end{equation*}%
Attaching 3-cells to $Y$ along these $f_{i}$ $(1\leq i\leq n)$, we obtain a
3-dimensional CW complex $Y\cup _{i=1}^{n}e_{i}^{3}$. Let $i:X\overset{i}{%
\rightarrow }X\vee (S^2)^{\vee n}$ be the natural inclusion. By our
construction, the canonical composition 
\begin{equation*}
f^{\prime }:X\overset{i}{\rightarrow }X\vee (S^2)^{\vee n}\overset{f}{%
\rightarrow }Y\rightarrow Y\cup _{i=1}^{n}e_{i}^{3}
\end{equation*}%
induces isomorphisms on both $\pi _{1}$ and $\pi _{2}$ (the same as the
second homology groups of the universal covers). It is not hard to see that 
\begin{equation*}
H_{3}(\tilde{X})=H_{3}(\widetilde{Y\cup _{i=1}^{n}e_{i}^{3}})=0.
\end{equation*}
Therefore, $f^{\prime }$ induces a homotopy equivalence between the chain
complexes of the universal covering spaces. By the Whitehead theorem, $%
f^{\prime }$ is a homotopy equivalence.
\end{proof}

\begin{proof}[Proof of Theorem \protect\ref{th1} (i) and (ii)]
By Proposition \ref{prop}, (i) is equivalent to (ii). We prove (ii) as
follows. By a result of Mannan \cite{ma}, $X$ is the plus construction of a
finite 2-complex $Y$ with respect to a perfect normal subgroup $P\leq \pi
_{1}(Y).$ Therefore, we have a short exact sequence of groups 
\begin{equation*}
1\rightarrow P\rightarrow \pi _{1}(Y)\rightarrow \pi _{1}(X)\rightarrow 1.
\end{equation*}

Since $\pi _{1}(Y)/P=\pi _{1}(X)$ is finite and $Y$ is finite, the covering
space of $Y$ with fundamental group $P$ is again a finite CW complex. Hence $%
P$ is finitely generated. If the normal generation conjecture (Conjecture %
\ref{ng}) holds, $P$ is normally generated by a single element. Lemma \ref%
{lxba} says that $X\vee S^{2}$ is homotopy equivalent to a 2-dimensional CW
complex.
\end{proof}

\bigskip

Without the assumption of the Wiegold conjecture we only know that a finite
group G has property $D(2,n)$ for $n=\max \{1,1-\mathrm{def}(G)-\mu
_{2}(G)\},$ which follows the Swan-Jacobinski theorem in \cite{Jo} 29.3,
29.4 and Browning's results \cite{br}.

\section{Partial Euler characteristic and the Whitehead conjecture}

Recall definitions of $\mu _{n}(F)$ for an algebraic $n$-complex $F_{\ast }$
and $\mu _{n}(G)$ from Introduction. For a finitely presentable group $G$,
the following lemma follows from Swan \cite{sw} easily.

\begin{lemma}
\label{agif} Assume that $G$ is finitely presentable. The invariant $\mu
_{2}(G)$ can be realized by an algebraic 2-complex. In other words, there
exists an algebraic 2-complex 
\begin{equation*}
F_{2}\rightarrow F_{1}\rightarrow F_{0}\rightarrow \mathbb{Z}\rightarrow 0
\end{equation*}%
such that 
\begin{equation*}
\mu _{2}(G)=\dim _{\mathbb{Z}G}F_{2}+\dim _{\mathbb{Z}G}F_{0}-\dim _{\mathbb{%
Z}G}F_{1}.
\end{equation*}
\end{lemma}

\begin{proof}
It is enough to notice that $\mu _{2}(G)$ is finite by Theorem 1.2 in \cite%
{sw}.
\end{proof}

\bigskip

\begin{proof}[Proof of Theorem \protect\ref{th1} (iii)]
We prove a more general result: if a finitely presentable group $G$
satisfies the $D(2,n)$ problem, then 
\begin{equation*}
\mathrm{def}(G)\geq (1-n)-\mu _{2}(G).
\end{equation*}
By Lemma \ref{agif}, we can choose an algebraic 2-complex 
\begin{equation*}
(F_{\ast }):F_{2}\rightarrow F_{1}\rightarrow F_{0}\rightarrow \mathbb{Z}%
\rightarrow 0
\end{equation*}%
such that 
\begin{equation*}
\mu _{2}(G)=\dim _{\mathbb{Z}G}F_{2}+\dim _{\mathbb{Z}G}F_{0}-\dim _{\mathbb{%
Z}G}F_{1}.
\end{equation*}%
Since every algebraic 2-complex is geometric realizable by a $3$-dimensional
CW complex (cf. Johnson \cite{Jo} Theorem 60.2), there is a finite
3-dimensional CW complex of cohomological dimension 2 such that the reduced
chain complex 
\begin{equation*}
C_{2}(\tilde{Y})/C_{3}(\tilde{Y})\rightarrow C_{1}(\tilde{Y})\rightarrow 
\mathbb{Z}\pi _{1}(Y)\rightarrow \mathbb{Z}\rightarrow 0
\end{equation*}%
is homotopy equivalent to $(F_{\ast }).$ Assuming that $G$ has the $D(2,n)$
property, the wedge $X\vee (S^{2})^{\vee n}$ is homotopy equivalent to a
2-dimensional CW complex, which gives a presentation of $G.$ This implies
that $\mu _{2}(G)+n\geq 1-\mathrm{def}(G),$ \textsl{i.e.} $\mathrm{def}%
(G)\geq (1-n)-\mu _{2}(G).$ When Wiegold's Conjecture holds, the complex $X$
has property $D(2,1),$ which gives (iii).
\end{proof}

\bigskip

It is possible to place $\mu _{2}(G)$ in the broader setting of $(G,n)$%
-complexes, as follows (\textit{cf.} \cite{Ha}). Recall that a $(G,n)$%
-complex is a finite $n$-dimensional CW complex $X$ with fundamental group $%
G $ and vanishing homotopy group $\pi _{i}(X)=0$ for $i=2,3,\cdots ,n-1$. In
particular, a $(G,2)$-complex is a usual finite 2-dimensional CW complex
with fundamental group $G$.

\begin{definition}
Let $G$ be a finitely presentable group. Define 
\begin{equation*}
\mu _{n}^{g}(G)=\min \{(-1)^{n}\chi (X)\mid X\text{ is a }(G,n)\text{-complex%
}\}.
\end{equation*}%
If there is no such $X$ exists, define $\mu _{n}^{g}(G)=+\infty .$ We call
that a $(G,n)$-complex $X$ with $(-1)^{n}\chi (X)=\mu _{n}^{g}(G)$ is a
complex realizing $\mu _{n}^{g}(G)$.
\end{definition}

A few observations are immediate. It is clearly true that $\mu _{n}(G)\leq
\mu _{n}^{g}(G).$ Therefore, $\mu _{n}^{g}(G)>-\infty $ since $\mu
_{n}(G)>-\infty $ (\textit{cf.} Swan \cite{sw}). Moreover, $\mu _{2}(G)=\mu
_{2}^{g}(G)$ if and only if $\mu _{2}(G)=1-\mathrm{def}(G).$

Now we study the partial Euler characteristic and deficiency for groups of
low geometric dimensions. Recall that for a group $G$, the classifying space 
$\mathrm{B}G$ of $G$ is defined as the connected CW complex with $\pi _{1}(%
\mathrm{B}G)=G$ and $\pi _{i}(\mathrm{B}G)=0,i\geq 2$. It is unique up to
homotopy.

\begin{theorem}
\label{lgba} Let $G$ be a group having a finite $n$-dimensional classifying
space $\mathrm{B}G.$ We have the following.

\begin{enumerate}
\item[(i)] $\mu _{n}(G)=\mu _{n}^{g}(G);$ In particular, $\mu _{2}(G)=1-%
\mathrm{def}(G)$ if $G$ has a finite $2$-dimensional $\mathrm{B}G;$

\item[(ii)] Any finite CW complex $X$ with $\pi_1(X)=G$ satisfying the
following properties:

\begin{description}
\item[a)] the dimension is at most $n+1;$

\item[b)] the cohomological dimension $\mathrm{cd}(X)$ is at most $n;$

\item[c)] if $n\geq 3,$ the homotopy group $\pi _{i}(X)=0$ for $2\leq i\leq
n-1;$

\item[d)] $(-1)^{n}\chi (X)=\mu _{n}^{g}(G),$
\end{description}

is homotopy equivalent to $\mathrm{B}G.$
\end{enumerate}
\end{theorem}

\begin{proof}
Let $\mathrm{E}G$ be the universal cover of $\mathrm{B}G$. Since $\mathrm{E}%
G $ is contractible, one obtains the exact cellular chain complex of $%
\mathrm{E}G$: 
\begin{equation*}
C_{\ast }(\mathrm{E}G):0\rightarrow C_{n}(\mathrm{E}G)\rightarrow C_{n-1}(%
\mathrm{E}G)\ldots \rightarrow \mathbb{Z}G\rightarrow 0.
\end{equation*}%
This gives a (truncated) free resolution of $G$. In order to prove (i), it
suffices to show that this resolution gives the minimal Euler characteristic 
$\mu _{n}(G)$ since we notice earlier that $\mu _{n}(G)\leq \mu _{n}^{g}(G)$.

Suppose that $\mu _{n}(G)$ is obtained from the following partial resolution
of finitely generated free $\mathbb{Z}G$-modules: 
\begin{equation*}
F_{\ast }:F_{n}\overset{d}{\rightarrow }F_{n-1}\ldots \rightarrow
F_{1}\rightarrow \mathbb{Z}G\rightarrow 0.
\end{equation*}%
We claim that $F_{\ast }$ is exact at $F_{n},$ \textsl{i.e.} $\ker d=0.$
Once this is proved, $C_{\ast }(\mathrm{E}G)$ and $F_{\ast }$ are chain
homotopic to each other and hence have the same Euler characteristic.

To prove the claim, let $J$ be the kernel of $d$. By Schanuel's lemma, there
is an isomorphism 
\begin{equation*}
J\oplus C_{n}(\mathrm{E}G)\oplus F_{n-1}\ldots \cong F_{n}\oplus C_{n-1}(%
\mathrm{E}G)\ldots .
\end{equation*}%
Applying the functor $-\otimes _{\mathbb{Z}G}\mathbb{Z}$ to both sides of
this isomorphism, we see that $\mu _{n}(F)=(-1)^{n}\chi (\mathrm{B}G)$ and $%
J\otimes _{\mathbb{Z}G}\mathbb{Z}=0$ by noticing the fact that the complex $%
F_{\ast }$ attains minimal Euler characteristic after multiplying $(-1)^{n}$
among all the algebraic $n$-complexes. This implies that $C_{n}(\mathrm{E}%
G)\oplus F_{n-1}\ldots $ and $F_{n}\oplus C_{n-1}(\mathrm{E}G)\ldots $ have
the same finite free $\mathbb{Z}G$-rank. By Kaplansky's theorem, $J$ is the
trivial $\mathbb{Z}G$-module(\textit{cf.} \cite{Ka}, p. 328). This proves
(i).

We now prove (ii). Let $C_{\ast }(\tilde{X})$ be the chain complex of the
universal covering space of $X.$ Since \textrm{cd}$(X)\leq n,$ $C_{n+1}(%
\tilde{X})$ is a direct summand of $C_{n}(\tilde{X}),$ by the same argument
given in Lemma \ref{2.2}. Let $F^{1}$ be the chain complex%
\begin{equation*}
F_{\ast }^{1}:C_{n}(\tilde{X})/C_{n+1}(\tilde{X})\overset{d}{\rightarrow }%
C_{n-1}(\tilde{X})\rightarrow \cdots \rightarrow C_{1}(\tilde{X})\rightarrow 
\mathbb{Z}G\rightarrow 0.
\end{equation*}%
It is not hard to see that $\pi _{n}(X)\cong \ker d.$ Note that 
\begin{equation*}
\mu _{n}(F^{1})=(-1)^{n}\chi (X)=\mu _{n}(G).
\end{equation*}%
By the same argument as the first part of the proof, we get $\ker d=0.$ This
implies that $\tilde{X}$ is $n$-connected. Since $H_{n+1}(\tilde{X})=0,$ $%
\tilde{X}$ is contractible and $X$ is homotopy equivalent to $\mathrm{B}G.$
\end{proof}

\begin{remark}
Under the condition of Theorem \ref{lgba}, Harlander and Jensen \cite{Ha}
already prove that a $(G,n)$-complex realizing $\mu _{n}^{g}(G)$ is homotopy
equivalent to $\mathrm{B}G.$ Note that a $(G,n)$-complex is a special case
of $X$ in Theorem \ref{lgba}.
\end{remark}

We conclude with an application. Suppose that $G$ is a finitely presentable
group and 
\begin{equation*}
\mathbf{P}=\langle x_{1},\cdots ,x_{n}\mid r_{1},\cdots ,r_{m}\rangle 
\end{equation*}%
is a presentation of $G$. Denote by $G_{\mathbf{P}}$ the group given by the
presentation $\mathbf{P}.$ From each finite $2$-dimensional CW complex $X$,
one shrinks a spanning tree in the 1-skeleton to make $X$ have only a single
0-cell and obtains a finite presentation of $\pi _{1}(X)$. Namely, the
1-cells correspond one-one to a set of generators while the 2-cells
correspond one-one to a set of relators. Therefore, any counter-example to
the Whitehead conjecture gives rise to a 2-complex with a single 0-cell. For
a presentation $\mathbf{P},$ we will denote by $\chi (\mathbf{P})=m-n+1.$ A
sub-presentation of $\mathbf{P}=\langle x_{1},\cdots ,x_{n}\mid r_{1},\cdots
,r_{m}\rangle $ is a presentation $\langle y_{1},\cdots ,y_{n^{\prime }}\mid
s_{1},\cdots ,s_{m^{\prime }}\rangle $ with each $y_{i}\in \{x_{1},\cdots
,x_{n}\}$ and each $s_{i}\in \{r_{1},\cdots ,r_{m}\}$ is only a word of $%
y_{1},\cdots ,y_{n^{\prime }}.$

\begin{lemma}
\label{spog} Suppose that $\mathbf{P}^{\prime }=\langle y_{1},\cdots
,y_{n^{\prime }}\mid s_{1},\cdots ,s_{m^{\prime }}\rangle $ is a
sub-presentation of $\mathbf{P}=\langle x_{1},\cdots ,x_{n}\mid r_{1},\cdots
,r_{m}\rangle $ of a group $G_{\mathbf{P}}$. If $\mathbf{P}^{\prime \prime }$
is another finite presentation of $G_{\mathbf{P^{\prime }}}$, then one can
obtain a presentation of $G_{\mathbf{P}}$ from $\mathbf{P}^{\prime \prime }$
by adding $n-n^{\prime }$ generators and $m-m^{\prime }$ relations. In
particular, if $\mathbf{P}$ realizes $\mu _{2}^{g}(G_{\mathbf{P}})$, then $%
\mathbf{P}^{\prime }$ realizes $\mu _{2}^{g}(G_{\mathbf{P}^{\prime }})$.
\end{lemma}

\begin{proof}
Re-indexing and re-naming if necessary, we assume that 
\begin{equation*}
y_{1}=x_{1},\cdots ,y_{n^{\prime }}=x_{n^{\prime }},n^{\prime }\leq n
\end{equation*}%
and 
\begin{equation*}
s_{1}=r_{1},\cdots ,s_{m^{\prime }}=r_{m^{\prime }},m^{\prime }\leq m.
\end{equation*}%
It is clear that the words corresponding to $s_{1},\cdots ,s_{m^{\prime }}$
do not involve $x_{n^{\prime }+1},\cdots ,x_{n}$. If 
\begin{equation*}
\mathbf{P}^{\prime \prime }=\langle y_{1}^{\prime },\cdots ,y_{u}^{\prime
}\mid s_{1}^{\prime },\cdots ,s_{v}^{\prime }\rangle
\end{equation*}%
is another presentation of $G_{\mathbf{P}^{\prime }}$, we form a group $%
G^{\prime \prime }$ with the presentation 
\begin{equation*}
\langle y_{1}^{\prime },\cdots ,y_{u}^{\prime },x_{n^{\prime }+1},\cdots
,x_{n}\mid s_{1}^{\prime },\cdots ,s_{v}^{\prime }\rangle
\end{equation*}%
by adding $n-n^{\prime }$ free generators to $\mathbf{P}^{\prime \prime }$.
For each $1\leq i\leq n^{\prime },$ the letter $x_{i},$ viewed as an element
in $G_{\mathbf{P}^{\prime }},$ has a lifting $w_{i}$ in the free group $%
\langle y_{1}^{\prime },\cdots ,y_{u}^{\prime }\rangle .$ In other words, we
choose $w_i$ on the generators $y_{1}^{\prime },\cdots ,y_{u}^{\prime }$
such that the bijection $x_i\mapsto w_i, 1\leq i\leq n^{\prime }$ induces an
isomorphism $G_{\mathbf{P}^{\prime}} \to G_{\mathbf{P}^{\prime \prime}}.$

For each $1\leq i\leq n,$ define the word $\omega _{i}$ of $\{y_{1}^{\prime
},\cdots ,y_{u}^{\prime },x_{n^{\prime }+1},\cdots ,x_{n}\}$ as 
\begin{equation*}
\omega _{i}=\left\{ 
\begin{array}{c}
w_{i},1\leq i\leq n^{\prime }; \\ 
x_{i},n^{\prime }<i\leq n.%
\end{array}%
\right.
\end{equation*}%
Denote by $\phi $ the bijection 
\begin{equation*}
\phi :\{x_{1},\cdots ,x_{n}\}\rightarrow \{\omega _{1},\cdots ,\omega _{n}\}
\end{equation*}%
given by $x_{i}\mapsto \omega _{i}.$ For each $m^{\prime }<i\leq m,$ write $%
r_{i}=\Pi _{j=1}^{k_{i}}x_{ij}$ as a reduced word of $\{x_{1},\cdots
,x_{n}\},$ where $x_{ij}\in \{x_{1}^{\pm },\cdots ,x_{n}^{\pm }\}.$ Let $%
r_{i}^{\prime }=\Pi _{j=1}^{k_{i}}\phi (x_{ij})$ be the corresponding word
of 
\begin{equation*}
\{y_{1}^{\prime },\cdots ,y_{u}^{\prime },x_{n^{\prime }+1},\cdots ,x_{n}\}.
\end{equation*}%
Let $K$ be the normal subgroup of $G^{\prime \prime }$ normally generated by
the $m-m^{\prime }$ elements $r_{m^{\prime }+1}^{\prime },\cdots
,r_{m}^{\prime }$. We obtain a short exact sequence of groups 
\begin{equation*}
1\rightarrow K\rightarrow G^{\prime \prime }\rightarrow G_{\mathbf{P}%
}\rightarrow 1,
\end{equation*}%
where the third arrow is induced by the map $G_{\mathbf{P}^{\prime
}}\rightarrow G_{\mathbf{P}}$ from the natural inclusions of generators and
relators. From this exact sequence, one obtains the desired presentation 
\begin{equation*}
\mathbf{P}_{0}=\langle y_{1}^{\prime },\cdots ,y_{u}^{\prime },x_{n^{\prime
}+1},\cdots ,x_{n}\mid s_{1}^{\prime },\cdots ,s_{v}^{\prime },r_{m^{\prime
}+1}^{\prime },\cdots ,r_{m}^{\prime }\rangle
\end{equation*}%
of $G_{\mathbf{P}}$.

Assume that $\mathbf{P}$ realizes $\mu _{2}^{g}(G_{\mathbf{P}})$, while a
sub-presentation $\mathbf{P}^{\prime }$ does not realize $\mu _{2}^{g}(G_{%
\mathbf{P}^{\prime }})$. Suppose that $\mu _{2}^{g}(G_{\mathbf{P}^{\prime
}}) $ is realized by a 2-dimensional complex $X,$ which gives a presentation 
$\mathbf{P}^{\prime \prime }.$ We obtain a new presentation $\mathbf{P}_{0}$
of $G_{\mathbf{P}}$ by adding relators and generators to $\mathbf{P}^{\prime
\prime }.$ However, 
\begin{equation*}
\chi (\mathbf{P}_{0})=\chi (\mathbf{P}^{\prime \prime })+m-m^{\prime
}-(n-n^{\prime })=\chi (\mathbf{P}^{\prime \prime })-\chi (\mathbf{P}%
^{\prime })+\chi (\mathbf{P})<\mu _{2}^{g}(G_{\mathbf{P}}).
\end{equation*}%
This is a contradiction to the fact that $\mathbf{P}$ realizes $\mu
_{2}^{g}(G_{\mathbf{P}})$. Therefore, $\mathbf{P}^{\prime }$ realizes $\mu
_{2}^{g}(G_{\mathbf{P}^{\prime }}).$
\end{proof}

\bigskip

Recall that a CW complex $X$ is aspherical if the universal cover $\tilde{X}$
is contractible. A famous conjecture of Whitehead says that any subcomplex $%
Y $ of an aspherical 2-dimensional complex $X$ is aspherical as well (for
more details, see the survey article \cite{Bo}). As an application of
results proved above, we give an equivalent condition of the asphericity of $%
Y$, as follows.

\begin{corollary}
\label{sxia} Suppose that $X$ is a finite aspherical 2-complex and $Y$ is a
subcomplex of $X$. We have the following.

\begin{enumerate}
\item[(i)] The complex $Y$ realizes $\mu _{2}^{g}(\pi _{1}(Y));$

\item[(ii)] The complex $Y$ is aspherical if and only if the fundamental
group $\pi _{1}(Y)$ has a finite classifying space $\mathrm{B}\pi _{1}(Y)$
of dimension at most 2.
\end{enumerate}
\end{corollary}

\begin{proof}
Since $X$ is aspherical, it realizes $\mu _{2}^{g}(\pi _{1}(X))$ by Theorem %
\ref{lgba}. Notice that $Y$ gives a presentation of $\pi _{1}(Y),$ which is
a sub-presentation of the presentation given by $X$. Lemma \ref{spog}
implies that $Y$ realizes $\mu _{2}^{g}(\pi _{1}(Y))$. This proves part (i).

If $Y$ is aspherical, it is $\mathrm{B}\pi _{1}(Y)$ and hence is of
dimension at most 2. Conversely, assume that $\pi _{1}(Y)$ has a finite
classifying space $\mathrm{B}\pi _{1}(Y)$ of dimension at most 2. By Theorem %
\ref{lgba}, all the $(\pi _{1}(Y),2)$-complexes realizing $\mu _{2}^{g}(\pi
_{1}(Y))$ are homotopic to $\mathrm{B}\pi _{1}(Y)$. Therefore, $Y$ is
aspherical by part (i).
\end{proof}

\bigskip

Corollary \ref{ncor} is Corollary \ref{sxia} (ii).

\bigskip

\noindent \textbf{Acknowledgements}

The second author is supported by Jiangsu Natural Science Foundation (No.
BK20140402) and NSFC (No.11501459,11771345,11771022).

\bigskip

Infinitus, Nanyang Technological University. E-mail: JIFENG@ntu.edu.sg

\bigskip

Department of Mathematical Sciences, Xi'an Jiaotong-Liverpool University,
111 Ren Ai Road, Suzhou, Jiangsu, China 215123. E-mail:
Shengkui.Ye@xjtlu.edu.cn

\end{document}